\documentclass[12pt,reqno]{amsart}
\usepackage{amsmath,amsfonts,euscript,amscd,amsthm,amssymb,upref,graphics,color,verbatim}
\usepackage[normalem]{ulem}
\usepackage{enumitem}
\usepackage{graphics}
\usepackage{epsfig}

\theoremstyle{plain}
\newtheorem{theorem}{Theorem}
\newtheorem{proposition}[theorem]{Proposition}
\newtheorem{lemma}[theorem]{Lemma}

\theoremstyle{definition}

\newtheorem{remark}[subsection]{Remark}

\newtheorem{nothing*}[subsection]{}

\newcommand{\rien}[1]{}

\newcommand{\C}{\ensuremath{\mathbb{C}}}
\newcommand{\B}{\ensuremath{\mathbb{B}}}
\newcommand{\N}{\ensuremath{\mathbb{N}}}

\def\e{\epsilon}

\renewcommand{\epsilon}{\varepsilon}
\renewcommand{\phi}{\varphi}
\renewcommand{\emptyset}{\varnothing}

\title{A transcendental H\'enon map with an oscillating wandering Short $\C^2$}

\author{Leandro Arosio$^{\dag}$, Luka Boc Thaler$^{\dag}$ and Han Peters}

\begin{document}

\thanks{$^{\dag}$  Supported by the SIR grant ``NEWHOLITE - New methods in holomorphic iteration'' no. RBSI14CFME}
\address{ L. Arosio: Dipartimento Di Matematica\\
Universit\`{a} di Roma \textquotedblleft Tor Vergata\textquotedblright\  \\
 Italy} \email{arosio@mat.uniroma2.it}
\address{L. Boc Thaler: Dipartimento Di Matematica\\
Universit\`{a} di Roma \textquotedblleft Tor Vergata\textquotedblright\  \\
 Italy} \email{Luka.Boc@fmf.uni-lj.si}
\address{ H. Peters: Korteweg de Vries Institute for Mathematics\\
University of Amsterdam\\
the Netherlands} \email{hanpeters77@gmail.com}

\begin{abstract}
Short $\mathbb C^2$'s were constructed in \cite{F} as attracting basins of a sequence of holomorphic automorphisms whose rate of attraction increases superexponentially. The goal of this paper is to show that such domains also arise naturally as autonomous attracting basins: we construct a transcendental H\'enon map with an oscillating wandering Fatou component that is a Short $\C^2$. The superexponential rate of attraction is not obtained at single iterations, but along consecutive oscillations.
\end{abstract}
\renewcommand{\baselinestretch}{1.07}
%%%%%%  TOPMATTER:   %%%%%%%%%%%%%%%%%%%%%%%%%

%\subjclass[2000]{32E20, 32E30, 32H02}
%\date{April 15, 2013}
%\keywords{}

\vfuzz=2pt
%%%%%%%%%%%%%%%%%%%%%%%%%%%%%%%%%%%%%%%%%%%%%%%%%%%%%%%%%%%%%%%%%%%
%%%%%%%%%%%%%%%%%%%%%%%%%%%%%%%%%%%%%%%%%%%%%%%%%%%%%%%%%%%%%%%%%%%
%%%%%%%%%%%%%%%%%%%%%%%%%%%%%%%%%%%%%%%%%%%%%%%%%%%%%%%%%%%%%%%%%%%

\vskip 1cm

\maketitle
\tableofcontents
\section{Introduction}

In \cite{F} Forn{\ae}ss proved the existence of so-called Short $\mathbb C^k$'s. Such domains are increasing union of balls whose Kobayashi metric vanishes identically, but which allow for the existence of non-constant bounded plurisubharmonic functions, and are thus not equivalent to $\mathbb C^k$. In the construction by Forn{\ae}ss the Short $\mathbb C^k$ arises as an attracting basin for a non-autonomous dynamical systems, given by compositions of a sequence of holomorphic automorphisms, see Proposition \ref{lem:short} below. %In the same paper Forn{\ae}ss proves that for any H\'enon map, the sublevel sets of its Green's function have a structure of a Short $\mathbb C^2$.
We prove here that a Short $\mathbb C^k$ can also arise naturally as a Fatou component of a single holomorphic automorphism.

It is clear that we cannot work with the basin of an attracting fixed point: such domains are always biholomorphic to $\mathbb C^k$ \cite{RR, Sternberg}. We note that the attracting basins of neutral or semi-neutral fixed points are  often known to be biholomorphic to $\mathbb C^k$, see for example \cite{Ueda} for the semi-parabolic case, and \cite{Weickert, Hakim} for maps tangent to the identity. In contrast, in the recent paper \cite{BRS} examples of attracting basins of neutral fixed points were constructed that are not biholomorphic to $\mathbb C^k$ but to $\mathbb C \times \mathbb (\C^\star)^{k-1}$. Whether the basin of attraction of a neutral fixed point can be a Short $\mathbb C^k$ is not known, but we consider it unlikely.

Instead we will prove that there exist holomorphic automorphisms with \emph{wandering} Fatou components equivalent to a Short $\mathbb C^2$. It immediately follows that such maps give rise to infinitely many disjoint Short $\mathbb C^2$'s, giving an alternative construction to an observation from \cite{BPV}. Recall that a Fatou component $\mathcal{F}_0$ of a map $F$ is wandering if  $F^k(\mathcal{F}_0)\cap F^j(\mathcal{F}_0)=\emptyset$ for all $k\neq j$. A wandering Fatou component $\mathcal{F}_0$ is oscillating if some subsequence $F^{k_n}$ has bounded orbits in $\mathcal{F}_0$, while a different subsequence has orbits converging to infinity.% orbit starting in $\Omega$ is unbounded and some other has a bounded subsequence.

The first construction of a holomorphic automorphism of $\C^2$ with a wandering domain is due to Forn\ae ss-Sibony \cite{FS}, although the  complex structure of the domain has not been studied in the literature.
Recently Arosio-Benini-Forn\ae ss-Peters \cite{ABFP} constructed a transcendental H\'enon map, i.e. a holomorphic automorphism of $\C^2$ of the form $F(z,w)=(f(z)+aw,az)$ with $f\colon \C\to \C$ a transcendental  function, which admits an oscillating wandering domain biholomorphic to $\C^2$. In this paper we will modify the construction from \cite{ABFP} to obtain the following:

\begin{theorem}\label{thm:main2}
There exist a transcendental H\'enon map with a oscillating wandering domain biholomorphic to a Short $\C^2$.
\end{theorem}

\begin{remark}
Our proof in fact guarantees that the constructed wandering Fatou component is biholomorphic  to one of the Short $\mathbb C^2$'s constructed by Forn{\ae}ss. Little has been written about the possible equivalence classes of Short $\mathbb C^k$'s, although it is clear from the results in \cite{F} that there are at least countably many. Recall from the recent paper of Forstneri\v{c} and the second author \cite{BTF} that there are in fact uncountably many equivalent classes of complex manifolds with vanishing Kobayashi pseudometric that can be written as an increasing union of balls; a more general notion sometimes taken as the definition of a Short $\mathbb C^k$.
\end{remark}

%Recall that a  Short $\C^2$ is a complex manifold exhausted by biholomorphic images of a ball, which has everywhere vanishing Kobayashi distance and carries a bounded non-constant plurisubharmonic function (and thus is not biholomorphic to $\C^2$).
We note that a transcendental H\'enon map $F$ has constant Jacobian determinant, thus the rate of contraction cannot increase, even as orbits escape to infinity. Using Runge approximation we recursively construct an oscillating orbit $P_n$ in such a way that the number of iterates of the consecutive oscillations increases superexponentially fast. This guarantees the rate of contraction to increase superexponentially fast when one considers the iterations from a point $P_{n_j}$ in one oscillation to a point $P_{n_{j+1}}$ in the next, allowing the construction of maps that are, after a suitable rescaling, arbitrarily close to the non-autonomous sequence of maps constructed by Forn{\ae}ss in \cite{F}.

We construct a sequence of corresponding balls $B(P_{n_j}, \beta_{n_j})$, with $\beta_{n_j} \rightarrow 0$, each mapped by $F^{n_{j+1} - n_j}$ into the next ball $B(P_{n_{j+1}}, \beta_{n_{j+1}})$. The construction guarantees that the \emph{calibrated} basin is a Short $\mathbb C^2$; see \cite{PW} and the appendix of this paper for a discussion of calibrated basins.

The calibrated basin is necessarily contained in the corresponding Fatou component $\mathcal{F}_0$. To prove that the calibrated basin is in fact equal to the Fatou component, we use a method introduced in \cite{ABFP}: we construct a non-positive plurisubharmonic function on $\mathcal{F}_0$ that is strictly negative at some point in the calibrated basin, and constantly equal to $0$ on its complement. It follows that the Fatou component cannot be larger than the calibrated basin. In the appendix we show how the plurisubharmonic method can be applied to arbitrary calibrated basins, clarifying the main result from \cite{PW}.

Let us point out that while working with transcendental H\'enon maps is more restrictive than working with arbitrary holomorphic automorphisms, it also has advantages. We work with a sequence of Runge approximations, and  obtain a sequence of automorphisms converging locally uniformly to a limit map $F$. When working with holomorphic automorphisms it is a priori not clear that the limit map is also an automorphism: it may not be surjective. However, for a convergent sequence of transcendental H\'enon maps, surjectivity of the limit map is automatic.

\section{Preparatory results}
If $(H_{n})_{n\geq 1}$ is a sequence of automorphisms of $\C^2$, then for all $0\leq n\leq m$ we denote
$$H_{m,n}:=H_{m}\circ \dots \circ H_{n+1}.$$
Notice that with these notations we have for all $n\geq 0$, $$H_{n+1,n}=H_{n+1}, \quad H_{n,n}={\sf id}.$$
If for all $n\geq 1$ we have $H_{n}(\B)\subset \B $, where $\B$ denotes the unit ball, then we define the {\sl basin} of the sequence $(H_n)$ as the domain
$$
\Omega_H:=\bigcup_{n\geq 0} H_{n,0}^{-1}(\B).
$$

\begin{lemma}\label{lem:compare}
To every finite family  $(F_1, \dots , F_n)$ of holomorphic automorphisms of $\C^2$  satisfying  $F_{j}(\B)\subset\subset \B $ for all $0\leq j\leq n$
we can associate $\e(F_1, \dots , F_n)>0$ such that the following holds:

Given any two sequences  $(H_{n})_{n\geq 1}$ and $(G_{n})_{n\geq 1}$ of holomorphic automorphisms of $\C^2$ satisfying $H_{n}(\B)\subset\subset \B $ and $G_{n}(\B)\subset \B $ for all $n\geq1$, and moreover satisfying
$$
\|H_{n}-G_{n}\|_{\overline\B}\leq \e(H_1, \dots , H_n),\quad \forall n\geq 1,
$$
the basins $\Omega_{G}$ and $\Omega_H$ are biholomorphically equivalent.
\end{lemma}
\begin{proof}
Let $(H_{n})_{n\geq 1}$ and $(G_{n})_{n\geq 1}$ be sequences as above. We will show how to choose the constants $(\e(H_1, \dots , H_n))_{n\geq 1}$  to obtain a biholomorphism between $\Omega_{G}$ and $\Omega_H$.

For all $n\geq 0$ denote $ U_n:=G_{n,0}^{-1}(\overline{\B})$, $V_n:=H_{n,0}^{-1}(\overline{\B})$, and define
$$
\Phi_n:=H_{n,0}^{-1}\circ G_{n,0}.
$$
 Notice that $\Phi_n$ is a holomorphic automorphisms of $\C^2$ satisfying $\Phi_n(U_n)=V_n$.
We have that $$\Phi_{n+1}\circ\Phi_n^{-1}=H_{n,0}^{-1}\circ (H_{n+1}^{-1}\circ G_{n+1})\circ H_{n,0}.$$
Thus
\begin{align*}
\|\Phi_{n+1}-\Phi_{n}\|_{ U_n}&=\|\Phi_{n+1}\circ \Phi_n^{-1}-{\sf id}\|_{V_n}\leq M_{n+1}\|H_{n+1}^{-1}\circ G_{n+1}-{\sf id}\|_{\overline\B}\\
&\leq M_{n+1}N_{n+1}\|H_{n+1}-G_{n+1}\|_{\overline\B}
\leq M_{n+1}N_{n+1}\e(H_1, \dots , H_{n+1}),
\end{align*}
where $M_{n+1}>0$ is the Lipschitz constant of $H_{n,0}^{-1}$ on $H^{-1}_{n+1}(\overline \B)$  and $N_{n+1}>0$ is the Lipschitz constant of $H_{n+1}^{-1}$ on $\overline\B$.
Hence if  for all $n\geq 0$ we choose  $$\e(H_1, \dots , H_{n+1})\leq \frac{1}{2^{n+1}M_{n+1}N_{n+1}},$$ the sequence  $(\Phi_n)$ converges to a holomorphic map $\Phi$ on $\Omega_G$.

Recall that there exist $\tilde{\delta}>0$ such that for every holomorphic map $F:\overline\B\rightarrow\C^2$ which satisfies $\|F-{\sf id}\|_{\overline\B}\leq \tilde\delta$ we have that the differential $d_0F$ is nonsingular.
By Hurwitz's theorem the map $\Phi$ is either injective or degenerate.
Assuming that
$$\e(H_1, \dots , H_{n+1})\leq \frac{\tilde\delta}{2^{n+1} M_{n+1}N_{n+1}}$$
we have, since  $\Phi_0:= {\sf id}$,
$$
\|\Phi-{\sf id}\|_{\overline\B}\leq \sum_{k=0}^{\infty}\|\Phi_{k+1}-\Phi_k\|_{\overline\B}\leq \tilde\delta,
$$
and thus $\Phi$ is injective.

To prove surjectivity, observe that since  for every $n\geq 1$  we have $H_{n}(\B)\subset\subset \B $ it follows that $V_{n-1}\subset{\rm int}( V_{n})$ hence there exist $\delta_n>0$ such that for every injective holomorphic  map $F\colon \overline \B\to \C^2$ which satisfies $\|F-H_{n,0}^{-1}\|_{\overline\B}\leq \delta_n$ it follows that $V_{n-1}\subset F(\overline \B)$. Clearly the choice of $\delta_n$ depends only on $H_1,\ldots, H_{n}$. Let us assume that
$$\e(H_1, \dots , H_{n+1})\leq \frac{\min_{1\leq j\leq n}\delta_j}{2^{n+1} M_{n+1}N_{n+1}}.$$
We have
\begin{equation}\label{injectivity}
\|\Phi-\Phi_n\|_{U_n}\leq\sum_{k=n}^{\infty}\|\Phi_{k+1}-\Phi_k\|_{U_n}\leq \delta_n
\end{equation}
or equivalently
$$
\|\Phi\circ G_{n,0}^{-1}-H_{n,0}^{-1}\|_{\overline\B}\leq \delta_n
$$
which implies $V_{n-1}\subset\Phi(U_n)$. Since this holds for every $n$ it follows that $\Phi\colon \Omega_G\to \Omega_H$ is surjective.

\end{proof}

The following result is slightly modified version of   \cite[Theorem 1.4]{F}, we sketch the proof for the reader's convenience.
\begin{proposition}\label{lem:short} Let $(d_k)$ be a sequence of integers with $d_k\geq 2$, and let us define  $H_n(z,w)=((\frac{z}{2})^{d_n}+2^{-d_n\cdots d_1}w,2^{-d_n\cdots d_1}z)$.
Then  $\Omega_H:=\bigcup_{k\geq 0} H_{k,0}^{-1}(\B)$ is a Short $\C^2$.
\end{proposition}
\begin{proof}
We prove the following:
\begin{enumerate}
\item $\Omega_H$ is non-empty, open, connected set in  $\C^2$,
\item $\Omega_H$ is the increasing union of its subdomains $\Omega_j:=H_{j,0}^{-1}(\B)$ which are biholomorphic to the unit ball $\B$,
\item the infinitesimal Kobayashi metric of $\Omega_H$ vanishes identically,
\item  there is a non-constant plurisubharmonic function $\psi$ on $\C^2$ satisfying $\Omega_H=\{\psi<0\}$.
\end{enumerate}

First observe that $H_k$ is fixing the origin and $H_k(\B)\subset\subset\B$ for every $k\geq 1$. Since $H_k$ are holomorphic automorphisms of $\C^2$ we immediately  obtain (1) and (2).

For (3) let us fix $(p,\zeta)$, where $p\in \Omega_H$ and $\zeta\in T_p\Omega_H$. Pick $R>0$. Since $p\in \Omega_H$ we have $p_k:=H_{k,0}(p)\rightarrow 0$. Next we denote $\zeta_k:= d_pH_{k,0}(\zeta)$. Since $(H_{k,0})_k$ converge on $\Omega_H$ to a constant map it follows that $\zeta_k\rightarrow 0$. Now we can find $k>0$ such that the map $\eta_k(w):=p_k+wR\zeta_k$ satisfies  $\eta_k(\mathbb{D})\subset \B$, where $\mathbb{D}\subset \C$ denotes the unit disk. Finally we define
$\eta:\mathbb{D}\rightarrow \Omega_H$ as $\eta(w):=H_{k,0}^{-1}\circ \zeta_k (w)$. It follows from the definition of $\eta$ that $\eta(0)=p$ and
$\eta'(0)=d_{p_k}H^{-1}_{k,0}(R\zeta_n)=R\zeta$, hence (3) is proved.

\medskip
It remains to show that also property (4) holds. Let us write $H_{k,0}=(h_1^k,h_2^k)$ and  $\eta_k:=2^{-d_k\cdots d_1}$. We define
$\phi_k(z):=\max\{|h_1^k|,|h_2^k|,\eta_k\}$ and
$$
\psi_k:=\frac{\log{\phi_k}}{d_k\cdots d_1}.
$$

Claim 1: {\it The functions $\psi_k$ converge on $\C^2$ to a plurisubharmonic function $\psi$.}

We show first that $\phi_{k+1}\leq 2\phi_k^{d_{k+1}}$ on $\C^2$. Assume that $\phi_k\leq 1$. Then
\begin{align*}
\phi_{k+1}(z)&=\max\left\{\left|\left(\frac{h^k_1}{2}\right)^{d_{k+1}} +\eta_{k+1}h_2^{k}\right|,\eta_{k+1}\left|h_1^k\right|,\eta_{k+1}\right\}\\
&\leq \max\{\phi_k^{d_{k+1}}+\eta_k^{d_{k+1}},\eta_k^{d_{k+1}}\}\\
&\leq 2\phi_k^{d_{k+1}}.
\end{align*}

If $\phi_k> 1$, then
\begin{align*}
\phi_{k+1}(z)&=\max\left\{\left|\left(\frac{h^k_1}{2}\right)^{d_{k+1}} +\eta_{k+1}h_2^{k}\right|,\eta_{k+1}\left|h_1^k\right|,\eta_{k+1}\right\}\\
&\leq \max\{\phi_k^{d_{k+1}}+\eta_k^{d_{k+1}}\phi_k,\eta_k^{d_{k+1}}\phi_k,\eta_k^{d_{k+1}}\}\\
&\leq \phi_k^{d_{k+1}}+\phi_k\\
&\leq 2\phi_k^{d_{k+1}}.
\end{align*}

This way we obtain $\psi_{k+1}\leq \psi_k+\frac{\log 2}{d_{k+1}\cdots d_{1}}$. Now define $\tilde{\psi}_k:=\psi_k+\sum_{j>k}\frac{\log 2}{d_j\cdots d_1}$ and  observe that $\tilde{\psi}_{k+1}\leq\tilde{\psi}_{k}$.
The plurisubharmonic functions $\tilde{\psi}_k$ form a monotonically decreasing sequence, whose limit $\psi$ is therefore also plurisubharmonic.  Since $\tilde{\psi}_k-\psi_k\to 0$, it follows that
$\psi_k\to \psi.$

Claim 2:  {\it We have $\Omega_H=\{\psi<0\}$}.

Let us assume that $\psi(z_0)<0$, then there exist $k>0$ and $s<0$ such that $\frac{\log\phi_k(z_0)}{d_k\cdots d_1}<s<0$, hence $\phi_k(z_0)<e^{s\cdot d_k\cdots d_1}$.
From the definition of $\phi_k$ it follows that $\|H_{k,0}(z_0)\|<2e^{s\cdot d_k\cdots d_1}$, hence $H_{k,0}(z_0)\rightarrow 0$ and therefore $z_0\in \Omega_H$.

Next we assume that $z_0\in\Omega_H$. Then there exist $k_0>0$ such that $H_{k,0}(z_0)\in\B$ for all $k>k_0$. But this implies that $\psi_k(z_0)<0$, hence $\psi(z_0)\leq 0$. Since $H_{k,0}(0)=0$ for all $k$ it follows that $\psi_k(0)=-\log 2$ for all $k$, hence $\psi(0)=-\log2<0$. We have seen that $\psi\leq0$ on $\Omega_H$ and $\psi<0$ at some point in $\Omega_H$, therefore  it follows from  the maximum principle  that $\psi <0$ on $\Omega_H$.

Claim 3: {\it $\psi$ is not constant on $\Omega_H$.}

Suppose that the contrary is true, i.e. $\psi|_{\Omega_H}\equiv s <0$. First observe that $\|H_{k,0}(8,0)\|\rightarrow\infty$, hence $\Omega_H$ is not all $\C^2$.
Pick $R>0$ such that $\B(0,R)\subset\Omega_H$ and there exist a point $p\in \partial\B(0,R)\cap\partial\Omega_H$. Since $\Omega_H=\{\psi<0\}$ it follows that
 $\psi(p)\geq0$. By subaveraging property of plurisubharmonic functions, $\psi(p)$ is bounded above by the average on any small ball $\B(p,\epsilon)$. If  $\epsilon$ small enough then $\psi=s<0$ on  more than one third of the ball $\B(p,\epsilon)$, and since $\psi$ is upper semicontinuous, this leads to a contradiction.

\end{proof}

%We also recall the following  lemma for the reader's convenience.

%\begin{lemma}(Lambda Lemma \cite{PM})\label{lem:lambda} Let G be a holomorphic automorphism of $\C^2$ with a saddle fixed point at the origin. Denote $W_G^s(0)$ and $W_G^u(0)$ the stable and unstable manifolds respectively. Let $p\in W_G^s(0)\backslash \{0\}$ and $q\in W_G^u(0)\backslash \{0\}$, and let $D(p)$ and $D(q)$ be holomophic disks through $p$ and $q$, respectively transverse to $W_G^s(0)$ and $W_G^u(0)$. Let $\e>0$. Then there exists $N\in \N$  and $N_1>2N+1$, and a point $x\in D(p)$ with $G^{N_1}(x)\in D(q)$ such that $\|G^n(x)-G^n(p)\|<\e$ and $\|G^{N_1-n}(x)-G^{N_1-n}(q)\|<\e$ for $0\leq n\leq N$, and $\|G^n(x)\|<\e$ when $N<n<N_1-N$.
%\end{lemma}

\section{Proof of   Theorem \ref{thm:main2}}

The proof of Theorem \ref{thm:main2} is based on the iterative construction contained in the following proposition, whose proof is postponed to Section \ref{inductionprop}.

\begin{proposition}\label{prop:henon} Let $a=\frac{1}{2}$. There exists a sequence of automorphims of $\C^2$
$$F_k(z,w)=(f_k(z)+aw,az),\qquad f_k(z)=z+O(z^2),\qquad k=0,1,2,\ldots$$
a sequence of points $P_n=(z_n,w_n)$ where $n=0,1,2,\ldots$, sequences  $\beta_n\rightarrow 0$,    $R_k\nearrow \infty$, strictly increasing sequences of integers  $(n_k)$ and $(N_k)$ satisfying $N_{k-1}\leq n_k\leq N_k$, and  a sequence of odd integers $(d_k)$, $d_k\geq 2$  such that the following properties are satisfied:

\begin{enumerate}
\item[(a)]   $\|F_{k}-F_{k-1}\|_{D(0,R_{k-1})\times\C}\leq 2^{-k}$ for all $k\geq 1$,
\item[(b)] $F_{k}(P_n)=P_{n+1}$ for all $0\leq n< N_k$,
\item[(c)] $\|P_{n_k}\|\leq \frac{1}{k}$ for all $k\geq 1$,
\item[(d)] $  |z_{N_{k}}|> R_{k}+3$ and $|z_n|+\beta_n< R_k$ for every $0\leq n<N_{k}$,
\item[(e)]   for all  $k\geq 1$, $\beta_n=\frac{1}{k}$ if $N_{k-1}<n< n_{k}$,
and $\beta_n=\frac{1}{k+1}$ if $n_k< n\leq N_k,$ and
$\beta_{n_k}< \frac{1}{k+1}$ is of the form $\beta_{n_k}=a^{2q_k}$, for some integer $q_k$,
\item[(f)]  $k|\log{\beta_{n_k}}|\leq d_k\cdots d_1$ for all $k\geq 1$,
\item[(g)] for all $0\leq s\leq k$,
$$F_{k}^j(\B(P_{n_s},\beta_{n_s}))\subset \subset\B(P_{n_s+j},\beta_{n_s+j}), \quad \forall\,1\leq j\leq N_{k}-n_s,$$
\item[(h)]  for all $ 1\leq s\leq k$,
$$
\|\Phi_{n_{s}}^{-1}\circ F_k^{n_{s}-n_{s-1}}\circ\Phi_{n_{s-1}}-H_{s}\|_{\overline{\B}}<\min\{\varepsilon(H_1,\dots, H_s),a^{d_s\cdots d_1}\},
$$
where   $\Phi_n:=P_n+z\cdot\beta_n$, $H_j(z,w)=((az)^{d_j}+a^{d_j\cdots d_1}w,a^{d_j\cdots d_1}z)$.
\end{enumerate}
\end{proposition}

Using this proposition we can now prove our main theorem.

\begin{proof}[Proof of Theorem \ref{thm:main2}] Let $(F_k)$ be a sequence of automorphisms of $\mathbb{C}^2$ satisfying conditions $(a)-(h)$ of  Proposition \ref{prop:henon}. The sequence $(F_k)$ converges uniformly on compact subsets to a transcendental H\'enon map $F\in{\rm Aut}(\C^2)$ with a saddle fixed point at the origin and with an unbounded orbit $(P_n)$, a  sequence
 $\beta_n\to 0$, a strictly increasing sequences of integers  $(n_k)$ and  a sequence of odd integers $(d_k)$, $d_k\geq 2$  such that the following properties are satisfied:
\begin{enumerate}
\item[(i)] $P_{n_k}\to 0$,
\item[(ii)] for all $k\geq 0$,
\begin{equation}\label{palle}
F^j(\mathbb{B}(P_{n_k},\beta_{n_k}))\subset \mathbb{B}(P_{n_k+j},\beta_{n_k+j}),\quad \forall\, j\geq 0.
\end{equation}
\item[(iii)]  $\lim_{k\rightarrow\infty}\frac{\log{\beta_{n_k}}}{d_k\cdots d_1}=0$,
\item[(iv)]  if for all  $k\geq 1$ we denote
  $$G_{k}:=\Phi_{n_{k}}^{-1}\circ F^{n_{k}-n_{k-1}}\circ\Phi_{n_{k-1}}\in {\rm Aut}(\C^2),$$
 then by combining conditions (g) and (h) it follows that $G_{k}(\B)\subset \B$ for all $k$, and
 \begin{equation}\label{saruman}
 \|G_{k}-H_{k}\|_{\overline{\B}}\leq\min\{\varepsilon(H_1,\dots, H_k),a^{d_k\cdots d_1}\},\quad \forall\,k\geq 1,
 \end{equation}
where $H_k$ denotes the holomorphic automorphism $$H_{k}(z,w):=((az)^{d_k}+a^{d_k\cdots d_1}w, a^{d_k\cdots d_1}z).$$
\end{enumerate}

It follows from Lemma \ref{lem:short}  that $\Omega_H:=\bigcup_{k\geq 0} H_{k,0}^{-1}(\B)$ is a Short $\C^2$.  By Lemma \ref{lem:compare} the basin $\Omega_G:=\bigcup_{k\geq 0} G_{k,0}^{-1}(\B)$ is biholomorphic to $\Omega_H$. Define
$$\Omega_F:=\bigcup_{k=0}^{\infty} F^{-n_k}(\mathbb{B}(P_{n_k},\beta_{n_k})).$$
Notice that $\Omega_F=\Phi_{0}(\Omega_G),$ and hence $\Omega_F$ is a Short $\C^k$.
We now show that $\Omega_F$ is contained in an oscillating Fatou component.

First of all,  the set $\Omega_F$ is contained in the Fatou set of $F$. Indeed, by the invariance of the Fatou set it is enough to prove that
for all $k\geq 0$, the ball $\B(n_k,\beta_{n_k})$ is contained in the Fatou set. But this follows from (\ref{palle}) since
the euclidean diameter of $F^j(\mathbb{B}(P_{n_k},\beta_{n_k}))$ is bounded for  $j\geq 0$ (in fact, it goes to 0 as $j\to \infty$).
For all $j\geq 0$ denote $\mathcal{F}_{n_j}$ the Fatou component containing $\B(P_{n_j},\beta_{n_j})$. Since $\Omega_F$ is connected, it is contained in the Fatou component $\mathcal{F}_{0}$.

Since $\beta_n\to 0$, by (\ref{palle}) and by the identity principle it follows that all limit functions on each $\mathcal{F}_{n_j}$ are constants.
We claim that for all $j\geq 0$,  if $i>j$, then $\mathcal{F}_{n_i}\neq \mathcal{F}_{n_j}$, which implies that they are all  oscillating wandering domains.
Assume by contradiction that $\mathcal{F}_{n_j}=\mathcal{F}_{n_i}$, and set $k:=n_i-n_j$.
Since the origin is a saddle point, there exists a neighborhood  $U$  of the origin that contains no periodic points of order less than or equal to $k$.
 Since the sequence $(P_n)$ oscillates, there exists a  subsequence  $(P_{m_\ell})$ of $(P_n)$ such that $P_{m_\ell}\to z \in U \setminus \{0\}.$ But then
$$F^{m_\ell-n_j}(P_{n_i})=F^{n_i-n_j}(P_{m_\ell})\to F^k(z)\neq z,$$
which contradicts $F^{m_\ell-n_j}(P_{n_j})=P_{m_\ell}\to z$.

  We complete the proof by showing that $\Omega_F=\mathcal{F}_0$. Suppose by contradiction that $\Omega_F\neq \mathcal{F}_0$.
For all $k\geq 1$, let us define the plurisubharmonic function $\psi_k$ as follows
$$
\begin{aligned}
\psi_k(z)& :=\frac{\log({\max\{\|F^{n_k}(z)-P_{n_k}\|,\beta_{n_k}\eta_k\}})}{d_k\cdots d_1}\\
& =\frac{\log({\max\{\|\Phi_{n_{k}}\circ G_{k,0}  \circ  \Phi_{0}^{-1}(z) -P_{n_k}\|,\beta_{n_k}\eta_k\}})}{d_k\cdots d_1},
\end{aligned}
$$
where $\eta_k:=a^{d_k\cdots d_1}$.

By (\ref{palle}) the limit functions of the sequence $F^k$ are constant on $\mathcal{F}_0$.   Since the sequence $(P_{n_k})$ is bounded,  it follows that for all compact subsets $K \subset  \mathcal{F}_0$ we have
$$
\|F^{n_k}(z)-P_{n_k}\|_K\to 0.
$$
As a consequence, we have that $(\psi_k)$ is locally uniformly bounded from above on $\mathcal{F}_0$ and that for all $z\in \mathcal{F}_0$,
$\limsup_{k\to\infty}\psi_{k}(z)\leq 0.$
For all $k\geq 1$ define the plurisubharmonic function as
$$\varphi_k(w):=\psi_k\circ \Phi_{0}(w)=\frac{\log({\max\{\|G_{k,0} (w)\|,\eta_k\}})+ \log{\beta_{n_k}}}{d_k\cdots d_1}.$$
The sequence $(\varphi_k)$ is clearly locally uniformly bounded from above on $\Phi_{0}^{-1}(\mathcal{F}_0)$ and
\begin{equation}\label{legolas}
\limsup_{k\to\infty}\varphi_{k}(z)\leq 0, \quad \forall z\in \Phi_{0}^{-1}(\mathcal{F}_0).
\end{equation}

We now show that   the sequence $(\varphi_k)$ converges on  $\Phi_{0}^{-1}(\mathcal{F}_0)$ to a  function $\varphi$ (which has to satisfy $\varphi\leq 0$ by (\ref{legolas})) which is plurisubharmonic on $\Omega_G=\Phi_{0}^{-1}({\Omega_F})$, strictly negative at the origin, and identically zero on $\Phi_{0}^{-1}({\mathcal{F}_0})\setminus \Omega_G$.

Once this is done, we conclude the proof in the following way: such function $\varphi$ is upper semicontinuous on the whole $\Phi_{0}^{-1}({\mathcal{F}_0})$ since it is identically zero on $\Phi_{0}^{-1}({\mathcal{F}_0})\setminus \Omega_G$, and thus it coincides with its upper   semicontinuous regularization, which is plurisubharmonic on $\Phi_{0}^{-1}({\mathcal{F}_0})$ since $(\varphi_k)$ is locally uniformly bounded from above on $\Phi_{0}^{-1}({\mathcal{F}_0})$ (see \cite[Proposition 2.9.17]{Klimek}). But this contradicts the maximum principle.

Since  $\lim_{k\rightarrow\infty}\frac{\log{\beta_{n_k}}}{d_k\cdots d_1}=0$ we only need to prove that  the sequence  $$\tilde \varphi_k(w):=\frac{\log({\max\{\|G_{k,0} (w)\|,\eta_k\}})}{d_k\cdots d_1}$$
converges to such a function $\varphi$.

Recall that $\|G_k-H_k\|\leq\eta_k$ on $\overline\B$ by (\ref{saruman}).  For all $w \in \Omega_G$ define $\theta_k(w):=\max\{\|G_{k,0} (w)\|,\eta_k\}$.
Recall that $\Omega_G=\bigcup_{j\in \N}G_{j,0}^{-1}(\B)$. Fix $j\geq 0$ and  let $w\in G_{j,0}^{-1}(\B)$. Then for all $k\geq j$ we have that
$G_{k,0}(w)\in \B$, and
\begin{align*}
\theta_{k+1}(w)&=\max\{\|G_{k+1,0} (w)\|,\eta_{k+1}\}\\
&\leq \max\{\|H_{k+1}(G_{k,0} (w))\|+\eta_{k+1},\eta_{k+1}\}\\
&\leq \max\{\|G_{k,0} (w)\|^{d_{k+1}}+ \eta_{k+1} \|G_{k,0} (w)\| +\eta_{k+1},\eta_{k+1} \}\\
&\leq \max\{\theta_k^{d_{k+1}}(w)+ 2\eta_{k}^{d_{k+1}} ,\eta_{k}^{d_{k+1}} \}\\
&\leq 3\theta_k^{d_{k+1}}(w).
\end{align*}

 Hence for $k\geq j$ we have, for all  $w\in G_{j,0}^{-1}(\B)$,
$$\tilde \varphi_{k+1}(w)\leq \tilde\varphi_k(w)+\frac{\log 3}{d_{k+1}\cdots d_1}.$$
This implies that the sequence of plurisubharmonic functions on $G_{j,0}^{-1}(\B)$
$$\tilde \varphi_k+\sum_{j>k}\frac{\log 3}{d_{j}\cdots d_1}$$
is monotonically decreasing, hence its limit $\varphi$ exists and is plurisubharmonic on  $G_{j,0}^{-1}(\B)$. Since this holds for every $j\geq 0$ we obtain that the limit $\varphi$ exists and is plurisubharmonic on $\Omega_G.$
Notice that $\varphi(0)=-\log 2<0$.
If $w\in \Phi_{0}^{-1}({\mathcal{F}_0})\setminus \Omega_G$,  then  we have that $\|G_{k,0}(w)\|\geq 1$ for all $k\in \N$, which implies
$\tilde\varphi_k(w)\geq 0$ and thus $\liminf_{k\to \infty}\tilde\varphi_k(w)\geq 0$.  But then  (\ref{legolas})  implies that $\tilde\varphi_k(w)$ converges to 0.

\end{proof}

\section{Proof of Proposition \ref{prop:henon}}\label{inductionprop}
We prove this proposition by induction on $k$. We start the induction by letting $R_0=1$, $n_0=N_0=0$, $\beta_0=1$, $q_0=0$, $P_0=(z_0,w_0)$ with $|z_0|>6$ and $F_0(z,w)=(z+aw,az)$, such that all conditions are satisfied for $k =0$. Let us suppose that conditions (a)---(h) hold for certain $k$, and proceed with the constructions satisfying the conditions for $k+1$.

By Lemma 7.5 from \cite{ABFP}, which relies on the Lambda Lemma, there exist a finite $F_k$ orbit $(Q_j)=(z_j',w_j')_{0\leq j \leq M}$ such that:
\begin{itemize}
\item $\|Q_{\ell}\|<\frac{1}{k+1}$ for some $0<\ell<M$,
\item for small enough $0<\theta<\frac{1}{k+2}$ the three disks
$$\overline{D}(z_{N_k},\beta_{N_k}),\qquad \overline{D}(w_0'/a, \theta), \qquad \overline{D}(z_M', \theta)$$
are pairwise disjoint, and  disjoint from the polynomially convex set
\begin{equation}\label{defk}
K:=\overline{D}(0,R_k)\cup \bigcup_{0\leq i<M} \overline{D}(z_i',\theta).
\end{equation}
\end{itemize}

%\begin{remark}{\color{red} It follows directly from the construction of a new orbit that $Q_j\neq P_{N_k}$ for all $0< j \leq M$ and $Q_{-1}:=F_k^{-1}(Q_0)\neq P_{N_k}$. Since $F_k$ is a one to one map and $F_k(Q_j)=Q_{j+1}$ for all $-1\leq j < M$ and $F_k(P_j)=P_{j+1}$ for all $0\leq j<N_k$ we can see that $\{Q_0,\ldots,Q_M\}\cap \{P_j\}_{0<j\leq N_k}=\emptyset$. Finally since $Q_j\neq P_{N_k}$ for all $0\leq j \leq M$ we also obtain $Q_j\neq P_{0}$ for all $0\leq j \leq M$. This proves that we have not produced any periodic cycles.}
%\end{remark}

By continuity of $F_k$ there exists
$0<s_\ell<\frac{\theta}{2}$ small enough such that  for all $0< j\leq M-\ell$, $$F_k^j(\B(Q_\ell,s_\ell))\subset \subset \B(Q_{\ell+j},\frac{\theta}{2})$$ and such that for all $0< j\leq \ell$,
$$F_k^{-j}(\B(Q_\ell,s_\ell))\subset\subset \B(Q_{\ell-j}, \frac{\theta}{2}).$$
Moreover we can assume that
\begin{equation}\label{aragorn}
 s_\ell=a^{2q_{k+1}}\quad {\mbox {\rm for some}}\quad q_{k+1}\geq \ell+2.
 \end{equation}
We define  $\Phi_{\ell}=Q_{\ell}+s_{\ell}\cdot z.$

\begin{figure}[t]
\centering
\label{fig:Detour}
\includegraphics[width=5in]{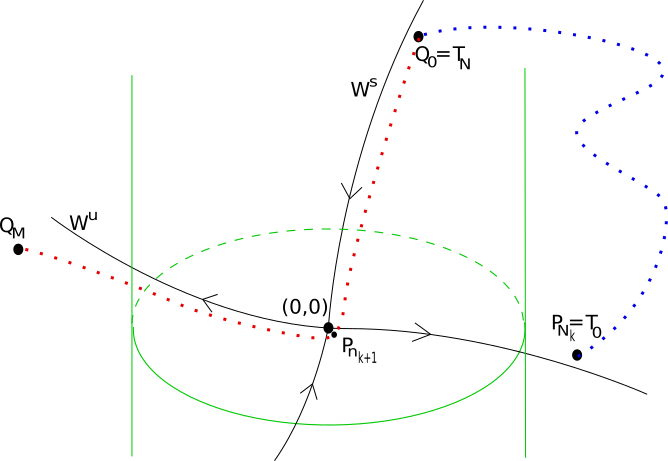}
\caption{A sketch of the piece of orbit constructed at step $k+1$. In green: the boundary of the cylinder $D(0, R_k) \times \mathbb C$. In red, the points $Q_n = (z_n^\prime, w_n^\prime)$ constructed using the Lambda Lemma. In blue, the points $T_n = (z_n^{\prime \prime}, w_n^{\prime \prime})$ connecting $T_0 = P_{N_k}$ to $T_N = Q_0$.}
\end{figure}

Our next goal is to construct the map $F_{k+1}$ with
a piece of orbit $T_0:=P_{N_k}, T_1,\dots , T_{N}:=Q_0$ such that
the iterate $F_{k+1}^{N}$ approximates the composition $F^{-\ell}_k\circ\Phi_{\ell}\circ H_{k+1}\circ \Phi_{n_k}^{-1}\circ F^{n_k -N_k}_k$ arbitrarily well near the point $P_{N_k}$, where
$$
H_{k+1}(z,w):=((az)^{d_{k+1}}+a^{d_{k+1}\cdots d_1}w, a^{d_{k+1}\cdots d_1}z),
$$
and where $d_{k+1}\geq 2$ is an odd integer to be determined later. See Figure \ref{fig:Detour} for a sketch of the piece of orbits constructed in step $k+1$.

\begin{lemma}\label{lem:henon} For every sufficiently large odd integer $d_{k+1}\geq 2$ the map $$F^{-\ell}_k\circ\Phi_{\ell}\circ H_{k+1}\circ \Phi_{n_k}^{-1}\circ F^{n_k -N_k}_k$$ can be written as a finite composition $\psi_{N}\circ\ldots\circ \psi_1$ of maps of the form $\psi_j(z,w)=(\varphi_j(z)+aw,az)$, with $\varphi_j\colon \C\to \C$ holomorphic.
\end{lemma}
\begin{proof}Let $F_k(z,w)=(f_k(z)+aw,az)$ and $\tau(z,w)=(aw,az)$. The map $\tau$ is clearly of the form $(\varphi(z)+aw,az)$ if we set $\varphi\equiv 0$.  A quick computation shows that
$$
F^{-1}_k(z,w) = \tau^{-2}\circ \tau\circ \Lambda\circ \tau \circ \tau^{-2} (z,w).
$$
where  $\Lambda(z,w)=(-f_k(z)+aw,az)$.
Clearly $\tau^{-2}$ commutes with $\tau$. The map $\tau^{-2}$ does not commute with $\Lambda$, but it satisfies the following relations:
for all $j\geq1$,
$$\tau^{-2j}\circ \Lambda=\tilde\Lambda_j\circ \tau^{-2j},\quad \Lambda\circ \tau^{-2j}=\tau^{-2j}\circ \hat\Lambda_j,$$
where we denote
$$
\tilde\Lambda_j(z,w):=(-\frac{f_k(a^{2j}z)}{a^{2j}}+aw,az),\quad \hat\Lambda_j(z,w):=(-a^{2j}f_k(\frac{z}{a^{2j}})+aw,az).$$
Notice that the maps $\tilde\Lambda_j$ and $\hat\Lambda_j$ are of the required form $(z,w)\rightarrow(\varphi(z)+aw,az)$.
If for all $j\geq 1$ we denote
$$\tilde\Lambda_{0,j}:=\tilde\Lambda_1\circ\ldots\circ \tilde\Lambda_j, \quad \hat\Lambda_{j,0}:= \hat\Lambda_{j}\circ\ldots\circ \hat\Lambda_1,$$
we obtain
\begin{equation}\label{bilbo}
F^{-j}_k=\tau\circ \tilde \Lambda_{0,j}\circ \tau \circ\tau^{-2(j+1)}=\tau^{-2(j+1)}\circ \tau  \circ \hat\Lambda_{j,0} \circ \tau.
\end{equation}

Define
$$\mathcal{H}_{k+1}(z,w):=(a^{(2-d_1\cdots d_{k+1})d_{k+1}}z^{d_{k+1}}+ aw,az),$$
and notice that
$$
H_{k+1}=\mathcal{H}_{k+1}\circ\tau^{d_1\cdots d_{k+1}-1}.
$$
We remark that from the last formula it follows that the product $d_1\cdots d_{k+1}$ must be odd, as otherwise the linear parts on left and right hand side cannot be equal.

Finally we focus on $\Phi_{\ell}$ and $\Phi_{n_k}^{-1}$.
By  inductive assumption (e) we know that    $\beta_{n_k}=a^{2q_k}$. By (\ref{aragorn}) we know that $s_{\ell}=a^{2q_{k+1}}$. Let us write $P_{n_k}=(z_{n_k}, w_{n_k})$ and $Q_\ell=(z_{\ell}', w_{\ell}')$  and recall that
$$\Phi_{\ell}(z,w)=(s_{\ell}\cdot z+z_{\ell}',s_{\ell}\cdot w+w_{\ell}').$$
If we define
$$S_1(z,w):=(\frac{w_{\ell}'}{a^{2(q_{k+1}-1)+1}}+aw,az),\quad S_2(z,w):=(\frac{z_{\ell}'}{a^{2(q_{k+1}-1)}}+aw,az),$$ then a quick computation gives us
$$
\Phi_{\ell}=\tau^{2(q_{k+1}-1)}\circ S_2\circ S_1.
$$
For the inverse
$$\Phi_{n_k}^{-1}(z,w)=(\frac{z-z_{n_k}}{\beta_{n_k}},\frac{w-w_{n_k}}{\beta_{n_k}}).$$
we denote $R_1(z,w):=(-a w_{n_k}+aw,az)$ and $R_2(z,w):=(-a^2 z_{n_k}+aw,az)$
$$
\Phi_{n_k}^{-1}=\tau^{-2(q_k-1)}\circ R_2\circ R_1.
$$

\medskip
Now we can write

$$F^{-\ell}_k\circ\Phi_{\ell}\circ H_{k+1}\circ \Phi_{n_k}^{-1}\circ F^{n_k-N_k}_k=
$$
$$
=\tau\circ \tilde\Lambda_{0,\ell}\circ \tau^{2(q_{k+1}-\ell-2)+1}\circ S_2\circ S_1\circ  \mathcal{H}_{k+1}\circ\tau^{d_1\cdots d_{k+1}-1-2(q_k-1)}\circ R_2\circ R_1\circ\tau^{-2(N_k-n_k+1)}\circ \tau  \circ \hat\Lambda_{N_k-n_k,0} \circ \tau
$$
\begin{equation}\label{gandalf}
=\tau\circ \tilde\Lambda_{0,\ell}\circ  \tau^{2(q_{k+1}-\ell-2)+1} \circ S_2\circ S_1 \circ  \mathcal{H}_{k+1}\circ\tau^{d_1\cdots d_{k+1}-1 -2(N_k-n_k+q_k)}\circ \tilde{R}_2\circ \tilde{R}_1\circ \tau  \circ \hat\Lambda_{N_k-n_k,0}  \circ \tau,
\end{equation}
where
\begin{align*}
\tilde{R}_1(z,w)&=(-a^{2(N_k-n_k+1)+1} w_{n_k}+aw,az),\\
\tilde{R}_2(z,w)&=(-a^{2(N_k-n_k+1)+2} z_{n_k}+aw,az).
\end{align*}

Since $q_{k+1}\geq \ell + 2$ it follows that $2(q_{k+1}-\ell-2)+1\geq 0$. If we choose the positive odd integer $d_{k+1}$ in such a way that
$$
d_{k+1}\geq \frac{2(N_k-n_k+q_k)+1}{d_1\cdots d_n},
$$ then $d_1\cdots d_{k+1}-1 -2(N_k-n_k+q_k)\geq 0$, and hence the lemma is proven.

\end{proof}

Let $\psi_{N}\circ\ldots\circ \psi_1$ be the maps given by  Lemma \ref{lem:henon}.
Let us denote $P_{N_k}:=(x_0,y_0)$, and, for all $1\leq n\leq N$,
 $$
X_n:=(x_n,y_n):=\psi_n\circ\ldots\circ \psi_1(x_0,y_0).
$$

Notice that  $X_N:= Q_0$, and that $y_n=ax_{n-1}$.

\begin{lemma}\label{lem:radius} Define $W:=F_k^{N_k-n_k}(\B(P_{n_k},\beta_{n_k}))$. If $d_{k+1}$ is sufficiently large  then for every $1\leq n\leq N$ we have
$$\psi_n\circ\ldots\circ \psi_1(W)\subset\subset  \B(X_n,\frac{1}{k+1}).$$
\end{lemma}

\begin{proof}
By condition (g) of the inductive process we have that for all $0\leq j\leq N_k-n_k$,
\begin{equation}\label{frodo}
F_k^{-j}(W)\subset \subset \B(P_{N_k-j},\frac{1}{k+1}).
\end{equation}
Notice that  $F_k^{-N_k+n_k}(W)=\B(P_{n_k},\beta_{n_k})$.

By (\ref{bilbo}) we have  $F_k^{-1}=\tau^{-4}\circ\tau\circ\hat \Lambda_{1,0}\circ\tau$.
Since $\tau(z,w)=(aw,az)$, it follows that $\psi_1(W)=\tau (W)\subset \B(X_1,\frac{1}{k+1})$.
By (\ref{frodo}) we obtain
$$\tau^{-3}\circ\hat  \Lambda_{1}\circ\tau (W)\subset\subset\B(P_{N_k-1},\frac{1}{k+1}),$$
and thus
$$\psi_2\circ \psi_1(W)=\hat  \Lambda_{1}\circ\tau (W)\subset \subset\tau^3(\B(P_{N_k-1},\frac{1}{k+1}))\subset \B(X_2, \frac{1}{k+1}).$$

Since  $F_k^{-2}= \tau^{-5}\circ\hat{\Lambda}_2\circ\hat{\Lambda}_1\circ\tau$, by (\ref{frodo}) we obtain
$$\tau^{-5}\circ\hat{\Lambda}_2\circ\hat{\Lambda}_1\circ\tau(W)\subset \subset\B(P_{N_k-2},\frac{1}{k+1}),$$ and thus
$$\psi_3\circ \psi_2\circ \psi_1(W)=\hat{\Lambda}_{2}\circ\hat{\Lambda}_{1}\circ\tau( W) \subset \subset \B(X_3,\frac{1}{k+1}).$$

Repeating this argument we obtain, for all $1\leq n\leq N_{k}-n_k+1$,
$$\psi_n\circ \dots\circ \psi_1(W)=\hat\Lambda_{n-1,0}\circ \tau(W)\subset \subset \B(X_n,\frac{1}{k+1}).$$

The maps $\tilde{R}_1$, $\tilde{R}_2$, $\tau$ are affine contractions, hence  after applying these maps the image of $W$ will still be contained in the respective balls of radius $\frac{1}{k+1}$. Let us look what happens if we now apply the map $\mathcal{H}_{k+1}$. Notice that
$$
\mathcal{H}_{k+1}\circ\tau^{d_1\cdots d_{k+1}-1 -2(N_k-n_k+q_k)}\circ \tilde{R}_2\circ \tilde{R}_1\circ \tau  \circ \hat{\Lambda}_{N_k-n_k,0} \circ \tau= H_{k+1}\circ \Phi_{n_k}^{-1}\circ F^{n_k-N_k}_k,
$$
hence
$$
\mathcal{H}_{k+1}\circ\tau^{d_1\cdots d_{k+1}-1 -2(N_k-n_k+q_k)}\circ \tilde{R}_2\circ \tilde{R}_1\circ \tau  \circ \hat{F}_{k,(N_k-n_k)} \circ \tau (W)= H_{k+1}(\B).
$$

By choosing $d_{k+1}$ sufficiently large we can make sure that
$H_{k+1}(\B)$ and all remaining images of $W$ are contained in the balls of radius  $\frac{1}{k+1}$.

\end{proof}

Let us denote $P_{N_k}:=(z_0'',w_0'')$,  $Q_0:=(z_N'',w_N'')$, let $z_1'',\ldots, z_{N-2}''\in \mathbb{C}$ be some sequence of points
 to be determined later, and let $z_{N-1}'':=\frac{w_{0}'}{a}=\frac{w_{N}''}{a}$. Next we define
 for all $0\leq n \leq N$, $w_{n}'':=az_{n-1}''$, $T_n:=(z_{n}'',w_{n}'')$ and  $\Theta_n(z,w):=(z-z_n''+x_n, w-w_n''+y_n)$. For all $1\leq n \leq N$,
\begin{align*}
G_n(z,w)&:=\Theta_{n}^{-1}\circ\psi_n\circ\Theta_{n-1}=\\
 &=(z_{n}''+\varphi_n(z-z_{n-1}''+x_{n-1})-\varphi_{n}(x_{n-1})-aw_{n-1}''+aw ,az),
\end{align*}
which is of the form $G_n(z,w)=(g_n(z)+aw,az),$
where
\begin{equation}\label{definizioneg}
g_n(z):=z_{n}''+\varphi_n(z-z_{n-1}''+x_{n-1})-\varphi_{n}(x_{n-1})-aw_{n-1}''.
\end{equation}
Notice that   for every $1\leq n \leq N$ we have $G_n(T_{n-1})=T_n,$
since $\Theta_j(T_j)=X_j$ for all $0\leq j\leq N$.
Notice also  that $\Theta_N=\Theta_0={\sf id}$, and thus
$$
G_N\circ\ldots\circ G_1=F^{-\ell}_k\circ\Phi_{\ell}\circ H_{k+1}\circ \Phi_{n_k}^{-1}\circ F^{n_k-N_k}_k.
$$

Since the maps $\Theta_{n}$ are just translations, and $\Theta_0$ is the identity, it follows from Lemma \ref{lem:radius} that
$$G_n\circ\ldots\circ G_1(W)=\Theta_n^{-1}\circ \psi_n\circ \dots \circ \psi_1\circ \Theta_0\subset\subset \B(T_n,\frac{1}{k+1}).$$
\begin{remark}
Observe that $G_N=\Theta_{N}^{-1}\circ\tau\circ\Theta_{N-1}=(a(y_{N-1}-w_{N-1}'') +aw,az)$. We know that $G_N\circ\ldots\circ G_1(W)\subset\subset \B(Q_0,\frac{\theta}{2})$, hence
$$G_{N-1}\circ\ldots\circ G_1(W)\subset\subset \B(T_{N-1},\theta).$$
\end{remark}
Fix $d_{k+1}$ such that  $\frac{\log{s_{\ell}}}{d_1\ldots d_{k+1}}\leq \frac{1}{k+1}$ and such that Lemma \ref{lem:henon} and Lemma \ref{lem:radius} hold. We  now have  complete freedom of choosing points $z_1'',\ldots, z_{N-2}''\in \mathbb{C}$ such that the disks
$$
 \overline{D}(z_{N_k},\beta_{N_k}),\qquad \overline{D}(w_0'/a, \theta), \qquad \overline{D}(z_M', \theta),\qquad  (\overline{D}(z_j'',\frac{1}{k+1}))_{1\leq j\leq N-2 }
$$
are pairwise disjoint, and  disjoint from the polynomially convex set $K$ defined by (\ref{defk}). Let H denote the union of these disks, and define $R_{k+1}>0$ sufficiently large so that $K\cup H\subset D(0,R_{k+1})$.

We define a holomorphic function $h$ on the polynomially convex set $K\cup H$ in the following way:
\begin{enumerate}
\item $h$ coincides with $f_k$ on $K$
\item  $h|_{\overline{D}(z_{N_k},\beta_{N_k})}$ coincides with $g_1$, where the $g_j$ are defined in (\ref{definizioneg}),
\item   $h|_{\overline{D}(z_j'',\frac{1}{k+1})}$ coincides with $g_{j+1}$ for all $1\leq j \leq N-2$,
\item $h|_{\overline{D}(w_0'/a, \theta)}$ coincides with  $g_N$,
\item $h|_{\overline{D}(z_M',\theta)}$ is constantly equal to some  value $A\in \C$
such that $ |A+aw'_M|>R_{k+1} +3$.
\end{enumerate}

By the Runge Approximation Theorem there exists a function $f_{k+1}\in\mathcal{O}(\C)$ satisfying:

\begin{enumerate}
\item $f_{k+1}(0)=h(0)=0$ and $f_{k+1}'(0)=h'(0)=1$,
\item $f_{k+1}$ coincides with $h$ on all the points $z_j,z_j', z_j''$,
%\item $f_{k+1}(z_j)=h(z_j)$ for all $0\leq j< N_k$,
%\item $f_{k+1}(z_j')=h(z_j')$ for all $0\leq j\leq M$,
%\item $f_{k+1}(z_j'')=h(z_j'')$ for all $0\leq j\leq N$,
\item $\|f_{k+1}-h\|_{K\cup H}<\delta_{k+1}\leq 2^{-k-1}$, with $\delta_{k+1}$ to be chosen later.
\end{enumerate}
We define $F_{k+1}:=(f_{k+1}(z)+aw,az)$, so that the sequences of points
$$
(P_j)_{0\leq j\leq N_k},\qquad (T_j)_{1\leq j\leq N-1},\qquad (Q_j)_{0\leq j\leq M}
$$
together form the start of an $F_{k+1}$-orbit.

Set $n_{k+1}:=N_k+N+\ell$ and $N_{k+1}:=N_k+N+M+1$, $P_{N_{k+1}}:=F_{k+1}(Q_M)$. Define  $\beta_j:=\frac{1}{k+1}$ if $N_k\leq j< n_{k+1}$, $\beta_{n_{k+1}}:=s_\ell$ and $\beta_j:=\frac{1}{k+2}$ if $n_{k+1}<j\leq N_{k+1}$.

It is immediate that  conditions (a) --- (f) are satisfied for the $(k+1)$-th step. We claim that $\delta_{k+1}$ can be chosen sufficiently small enough such that conditions (g) and (h) are satisfied for the $(k+1)$-th step.
We start with (g), that is we show that
for all $0\leq s\leq k+1$,
\begin{equation}\label{galadriel}
F_{k+1}^j(\B(P_{n_s},\beta_{n_s}))\subset \subset\B(P_{n_s+j},\beta_{n_s+j}), \quad \forall\,1\leq j\leq N_{k+1}-n_s,
\end{equation}
We have
\begin{align*}
 &F_k^{j}(\B(P_{n_k},\beta_{n_k})\subset \subset \B(P_{n_k+j},\beta_{n_k+j}),\quad \forall 1\leq j\leq N_{k}-n_k,\\
 &G_{j}\circ \dots \circ G_1\circ F_k^{N_k-n_k}(\B(P_{n_k},\beta_{n_k}))\subset \subset \B(T_j,\frac{1}{k+1}),\quad \forall 1\leq j\leq N-2,\\
 &G_{N-1}\circ \dots \circ G_1 \circ F_k^{N_k-n_k}(\B(P_{n_k},\beta_{n_k}))\subset \subset \B(T_{N-1},\theta),\\
 &F_k^j\circ G_{N}\circ \dots \circ G_1\circ  F_k^{N_k-n_k}(\B(P_{n_k},\beta_{n_k}))\subset \subset \B(Q_j,\theta), \quad \forall 0\leq j\leq M.\\
 \end{align*}
 Notice that all these sets are contained in $(K\cup H)\times \C$, and that
 $$(A+aw,az)\circ F_k^M\circ G_{N}\circ \dots \circ G_1\circ  F_k^{N_k-n_k}(\B(P_{n_k},\beta_{n_k}))\subset \subset \B(P_{N_{k+1}},\frac{1}{k+2}).$$
Hence we can choose $\delta_{k+1}>0$ small enough such that (\ref{galadriel}) holds for $s=k$.
Similarly one obtains (\ref{galadriel}) for  $0\leq s\leq k+1$, and hence (h), completing the proof of Proposition \ref{prop:henon}.

\section{Appendix: Calibrated basins and the plurisubharmonic method}

As a further illustration of using the plurisubharmonic method to prove that an attracting basin equals the Fatou component containing the basin, we take a closer look at the calibrated basins constructed in \cite{PW}. Let $f_0, f_1,f_2 \ldots$ be a sequence of automorphisms of $\mathbb C^k$, all having an attracting fixed point at the origin.
For all $j\geq 0$ there exists a radius $r_j>0$, a constant $0<\mu_j<1$ and a constant $C_j>0$ such that
$$\|f_j^n(z)\|\leq C_j\mu_j^n\|z\|,\quad \forall z\in B(0,r_j), n\geq 0. $$

We can choose $n_j$ large enough to obtain
$$
f_j^{n_j} (B(0, r_j)) \subset  B(0,r_{j+1}),
$$
and  $r_j\to 0$.
We then define the \emph{calibrated basin} $\Omega_{(r_j), (n_j)}$ by
$$
\Omega_{(r_j),(n_j)} := \bigcup_{j \in \mathbb N} (f_{j-1}^{n_{j-1}}\circ\dots \circ f_0^{n_0})^{-1} (B(0, r_j)).
$$
 It is easy to see that the calibrated basin may depend on the sequences $(r_j)$ and $(n_j)$ chosen.

Recall the following result from \cite{PW}:
\begin{theorem}  Fix the sequence $(r_j)$.
For $n_0, n_1, \ldots$ sufficiently large, where each $n_j$ may depend on the choices of $n_0, \ldots , n_{j-1}$, the calibrated basin of the sequence $f_0^{n_0}, f_1^{n_1}, \ldots$ is biholomorphic to $\mathbb C^k$.
\end{theorem}
Recall that it may be necessary to replace the maps $f_j$ by large iterates: all assumptions (with $r_j = \frac{1}{2}$) are satisfied by the maps in \cite{F}, see Proposition \ref{lem:short}, but in this case the calibrated basin is a Short $\mathbb C^k$, and hence not equivalent to $\mathbb C^k$.

One may wonder whether it is necessary to work with the calibrated basin instead of the basin that contains all points whose orbits converge to $0$. It turns out that this may indeed be necessary: for suitable choices of the sequence $f_0, f_1, \ldots$ the full basin may not be open, even when replacing the maps with arbitrarily high iterates $f_0^{n_0}, f_1^{n_1}, \ldots$, see \cite{PW}. This raises another natural question: is the calibrated basin equal to the Fatou component $\mathcal{F}_0$ containing the origin, that is, the largest connected open set with locally uniform convergence to $0$? Here we prove that this is indeed the case.

\begin{theorem}
Fix the sequence $(r_j)$. For $n_0, n_1, \ldots$ sufficiently large we have
$$
\mathcal{F}_0 = \Omega_{(r_j),(n_j)}.
$$
\end{theorem}

\begin{proof}
Define
$$
G_j(z) :=  \frac{\log \| f_j^{n_j} \circ \cdots \circ f_0^{n_0}(z)\|}{-n_j \log \mu_j}.
$$
Observe that, given a compact $K\subset \mathcal{F}_0$, for big $j$ we have that  $f_j^{n_j} \circ \cdots \circ f_0^{n_0}(K)\subset \B^k$, and thus that $G_j|_K\leq 0$.
Let
$$
G = \limsup_{j\to \infty} G_j,
$$
and write $G^\star$ for the upper semi-continuous regularization of $G$. It follows that $G^\star$ is plurisubharmonic on  $ \mathcal{F}_0$.

If we choose $n_j$ sufficiently large we may assume  that $$\frac{\log r_j}{ n_j\log\mu_j}\longrightarrow 0,$$ hence for any point $z \in \mathcal{F}_0 \setminus \Omega_{(r_j),(n_j)}$ we have $G^\star(z) = G(z) = 0$.

On the other hand, let $z\in \Omega_{(r_j),(n_j)},$ and let $j\geq 0$ be large enough such that
$z_j:=f_{j-1}^{n_{j-1}} \circ \cdots \circ f_0^{n_0}(z)\in B(0,r_j)$.
Then $$G_j(z)=\frac{\log \|f_j^{n_j}(z_j)\|}{-n_j\log\mu_j}\leq \frac{\log C_j+\log \|z_j\|+n_j\log\mu_j}{-n_j\log\mu_j}.$$
Choosing the $n_j$'s large enough we thus obtain $G(z)\leq -1$ for all $z \in \Omega_{(r_j),(n_j)}$, which implies that  $G^\star(z)\leq -1$ for all $z \in \Omega_{(r_j),(n_j)}$. Since $\mathcal{F}_0$ is open and connected, it follows from the maximum principle that $\mathcal{F}_0 \setminus \Omega_{(r_j),(n_j)}$ must be empty, which completes the proof.
\end{proof}

\end{document}